\documentclass[12pt]{amsart}

\usepackage[margin=1.1in]{geometry}

\usepackage{amsmath}
\usepackage{amsfonts}
\usepackage{amssymb}
\usepackage{amsthm}
\usepackage{enumerate}
\usepackage{graphicx}
\usepackage{hyperref}
\usepackage{enumitem}

\input xy
\xyoption{all}

\theoremstyle{plain}
\newtheorem{thm}{Theorem}

\newtheorem{prop}[thm]{Proposition}
\theoremstyle{definition}

\newtheorem{rem}[thm]{Remark}

\newcommand{\F}{\mathbb{F}}

\newcommand{\aut}{\mathrm{Aut}}

\newcommand{\out}{\mathrm{Out}\,}
\newcommand{\inn}{\mathrm{Inn}\,}

\newcommand{\sln}{\mathrm{SL}_n}

\newcommand{\bb}[1]{\mathbb{#1}}

\makeatletter
\def\ps@pprintTitle{%
  \let\@oddhead\@empty
  \let\@evenhead\@empty
  \let\@oddfoot\@empty
  \let\@evenfoot\@oddfoot
}
\makeatother

\usepackage{marvosym}

\title{On homogeneous spaces with finite anti-solvable stabilizers}
\author{Giancarlo Lucchini Arteche}
\address{Departamento de Matem\'aticas, Facultad de
Ciencias, Universidad de Chile\\ Las Palmeras 3425, \~Nu\~noa, Santiago, Chile}
\email{luco@uchile.cl}

\thanks{The author was partially supported by Fondecyt Grant 1210010.}

\keywords{homogeneous spaces, rational points, non-abelian cohomology, finite simple groups}

\begin{document}

\begin{abstract}
We say that a group is anti-solvable if all of its composition factors are non-abelian. We consider a particular family of anti-solvable finite groups containing the simple alternating groups for $n\neq 6$ and all 26 sporadic simple groups. We prove that, if $K$ is a perfect field and $X$ is a homogeneous space of a smooth algebraic $K$-group $G$ with finite geometric stabilizers lying in this family, then $X$ is dominated by a $G$-torsor. In particular, if $G=\sln$, all such homogeneous spaces have rational points.\\

\noindent\textbf{MSC codes:} 14M17, 12G05, 14G05.
\end{abstract}

\maketitle

We say that a finite group $F$ is \emph{aut-split} if the exact sequence
\[1\to\inn(F)\to\aut(F)\to\out(F)\to 1,\]
splits, where $\inn(F)$ denotes the group of inner automorphisms of $F$ and $\out(F)$ is simply defined as the quotient $\aut(F)/\inn(F)$. A full characterization of aut-split finite simple groups was given in \cite{LMM}. These include the alternating group $A_n$ for $n=5$ or $n\geq 7$ and all 26 sporadic groups. As for groups of Lie type, some notation is needed. If $F$ is a finite simple group of Lie type over $\F_q$ with $q=p^m$ and $p$ prime, denote by $d$ the order of the group of diagonal automorphisms (i.e.~those induced by conjugation by diagonal matrices) modulo inner automorphisms. Then $F$ is aut-split if and only if one of the following holds:
\begin{itemize}
\item $F$ is a Chevalley group, not of type $D_{\ell}(q)$, and $(\frac{q-1}{d},d,m)=1$;
\item $F=D_\ell(q)$ and $(\frac{q^\ell-1}{d},d,m)=1$;
\item $F$ is a twisted group, not of type $^2D_{\ell}(q)$, and $(\frac{q+1}{d},d,m)=1$;
\item $F={}^2D_\ell(q)$ and either $\ell$ is odd or $p=2$.
\end{itemize}

We say that $F$ is \emph{anti-solvable} if all of its composition factors (which are simple groups) are non-abelian. In this note we are interested in anti-solvable groups whose composition factors are aut-split (i.e.~they belong to the list above). This includes of course the groups in the list themselves. The main result of this note is the following.

\begin{thm}\label{thm HS}
Let $K$ be a perfect field and let $G$ be a smooth algebraic $K$-group. Let $X$ be a homogeneous space of $G$ with geometric stabilizer $\bar F$. Assume that $\bar F$ is finite, anti-solvable, and that all of its composition factors are aut-split. Then $X$ is dominated by a $G$-torsor. In particular, if $G=\sln$, then $X$ has a rational point.
\end{thm}

Of course, the particular case of $G=\sln$ is a natural consequence of the triviality of torsors under $\sln$, which is valid over any field $K$. Now, recall that the \emph{Springer class}, which is class in a non-abelian 2-cohomology set associated to every homogeneous space $X$, corresponds to the obstruction to being dominated by a $G$-torsor (cf.~\cite[\S2]{DLA} or \cite[\S1]{FSS}, or \cite{SpringerH2} for the original article by Springer). Using this tool we immediately see that Theorem \ref{thm HS} is implied by the following result on non-abelian Galois cohomology.

\begin{thm}\label{thm H2}
Let $K$ be a perfect field and let $L$ be a finite $K$-lien (or $K$-kernel) whose underlying group is anti-solvable and all of its composition factors are aut-split. Then $H^2(K,L)$ has exactly one class and it is neutral.
\end{thm}

The rest of this note is devoted to proving Theorem \ref{thm H2}. We start with the particular case where the group is itself aut-split (for instance, if there is only one composition factor).

\begin{prop}\label{prop aut-split}
Let $K$ be a perfect field and let $L$ be a finite $K$-lien (or $K$-kernel) whose underlying group is
aut-split. Then $H^2(K,L)$ has a neutral class (in particular, the $K$-lien $L$ is representable). If moreover $L$ has a trivial center, then this is the only class in $H^2(K,L)$.
\end{prop}

\begin{proof}
Let $F$ be the underlying finite group of $L$ and let $\Gamma_K$ be the absolute Galois group of $K$. By the definition of a $K$-lien we have a homomorphism $\kappa:\Gamma_K\to\out(F)$. Since $F$ is aut-split, we can compose this with a section $\out(F)\to\aut(F)$ and get a group action of $\Gamma_K$ on $F$. Since $F$ is finite, this defines a $K$-form of $F$ and, since this action lifts $\kappa$, it corresponds to a neutral class $\eta\in H^2(K,L)$. The second assertion follows immediately from \cite[Prop.~1.17]{SpringerH2} (see also \cite[IV, Thm.~8.8]{MacLane}).
\end{proof}

\begin{proof}[Proof of Theorem \ref{thm H2}]
Let $L$ be a $K$-lien and let $F$ be its underlying group as in Theorem \ref{thm H2}. Since $F$ is anti-solvable, it has trivial center and hence the $K$-lien $L$ is representable (cf.~\cite[IV, Thm.~8.7]{MacLane}). We fix then a class $\eta\in H^2(K,L)$, which is unique again by the triviality of the center and \cite[Prop.~1.17]{SpringerH2}. We need to prove that it is neutral.

If $F$ has no proper nontrivial characteristic subgroups (i.e.~it is \emph{characteristically simple}), then it is a direct product of copies of a single simple group (cf.~\cite[Ch.~2, Thm.~1.4]{Gorenstein}), which has to be aut-split then by hypothesis. Since every automorphism of such a product permutes its direct factors, it is easy to see that $F$ is aut-split as well. We are done then by Proposition \ref{prop aut-split}.

Assume now that $F$ has a proper nontrivial characteristic subgroup $N$. It is clear that both $N$ and $F/N$ also satisfy the hypotheses of Theorem \ref{thm H2}, so we argue by induction. Since $N$ is characteristic, the arrow $\aut(F)\to\aut(F/N)$ is well-defined and thus the homomorphism $\kappa:\Gamma_K\to\out(F)$ gives a homomorphism $\Gamma_K\to\out(F/N)$, which gives us a $K$-lien $L'$ whose underlying group is $F/N$. By induction hypothesis, we know that $H^2(K,L')$ has exactly one class $\eta'$ and that it is neutral. We interpret this in the language of group extensions (cf.~again \cite[\S2.2.2]{DLA} or \cite[1.18]{FSS} or \cite[1.13--14]{SpringerH2}). The class $\eta'$ corresponds to a split extension which, since the arrow $H^2(K,L)\to H^2(K,L')$ clearly sends $\eta$ to $\eta'$, fits into a commutative diagram
\[\xymatrix{
1 \ar[r] & F \ar[r] \ar[d] & \ar[r] E \ar[d]_\pi & \Gamma_K \ar[r] \ar@{=}[d] & 1 \\
1 \ar[r] & F/N \ar[r] & E' \ar[r] & \Gamma_K \ar[r]  \ar@/_1pc/[l]_{s} & 1.
}\]
Consider the preimage $\pi^{-1}(s(\Gamma_K))\subset E$, which fits into the following commutative diagram
\[\xymatrix{
1 \ar[r] & N \ar[r] \ar[d] & \ar[r] \pi^{-1}(s(\Gamma_K)) \ar[d] & \Gamma_K \ar[r] \ar@{=}[d] & 1 \\
1 \ar[r] & F \ar[r] & E \ar[r] & \Gamma_K \ar[r] & 1.
}\]
The top row corresponds to a class in $H^2(K,L'')$ for some $K$-lien $L''$ with underlying group $N$. By induction hypothesis again we know that this class is neutral and hence the extension is split, giving a section $\Gamma_K\to\pi^{-1}(s(\Gamma_K))$ and hence a section $\Gamma_K\to E$ by composition, proving that the class $\eta$ is neutral.
\end{proof}

\begin{rem} In \cite{Bercov}, Bercov characterizes anti-solvable finite groups whose composition factors are aut-split as iterated twisted wreath products of simple groups. Note however that we did not need this result in the proof above.
\end{rem}

\begin{rem}
Theorems \ref{thm HS} and \ref{thm H2} are false if we take away the aut-split hypothesis on the composition factors, already for simple groups. Indeed, consider the group $A_6$ and the order 2 element $\phi\in\out(A_6)$ coming from the nontrivial class of outer automorphisms of $S_6$. Taking $K=\bb R$, we can define a $K$-lien $L$ via the homomorphism $\kappa:\Gamma_{\bb R}\to\out(A_6)$ that sends the nontrivial element in $\Gamma_{\bb R}$ to $\phi$. By the same arguments used above we see that $H^2(K,L)$ has only one class $\eta$. But this class is not neutral, since this would amount to the existence of a lift of $\phi$ to an automorphism $\varphi\in\aut(A_6)$ of order 2, which does not exist. Using \cite[Cor.~3.3]{DLA} we may construct a homogeneous space $X$ of $\sln$ over $\bb R$ whose Springer class is $\eta$. This implies that $X$ is not dominated by a torsor.
\end{rem}

\end{document}